\documentclass[11pt,reqno]{amsart}
\usepackage{fullpage}

\usepackage{amsmath}
\usepackage{amsfonts}
\usepackage{amsthm}
\usepackage{amssymb}
\usepackage{color}
\usepackage{enumerate}
\usepackage{graphicx}
\usepackage{subfigure}
\usepackage{mathrsfs}
\usepackage{multicol}

\theoremstyle{plain}
\newtheorem{thm}{Theorem}[section]

\newtheorem{lem}[thm]{Lemma}

\theoremstyle{remark}

\numberwithin{equation}{section}

\newcommand{\E}{\mathbb{E}}

\newcommand{\N}{\mathbb{N}}
\newcommand{\bP}{\mathbb{P}}
\newcommand{\bQ}{\mathbb{Q}}
\newcommand{\R}{\mathbb{R}}

\newcommand{\Z}{\mathbb{Z}}

\newcommand{\srC}{\mathscr{C}}

\newcommand{\cA}{\mathcal{A}}
\newcommand{\cB}{\mathcal{B}}

\newcommand{\cE}{\mathcal{E}}
\newcommand{\cF}{\mathcal{F}}
\newcommand{\cR}{\mathcal{R}}
\newcommand{\cU}{\mathcal U}
\newcommand{\cJ}{\mathcal{J}}

\newcommand{\al}{\alpha}
\newcommand{\gam}{\gamma}

\newcommand{\ep}{\varepsilon}

\newcommand{\om}{\omega}

\newcommand{\Gam}{\Gamma}

\newcommand{\Om}{\Omega}

\newcommand{\ol}{\overline}

\newcommand{\co}{{\rm co}\,}

\author[W. Jing]{Wenjia Jing$^\ddag$}
\address{$^\ddag$ Department of Mathematics, 
	      The University of Chicago, 
	      5734 S. University Avenue Chicago, 
	      IL 60637, USA}
\email{wjing@math.uchicago.edu}
 
\author[P.E. Souganidis]{Panagiotis E. Souganidis$^{\sharp,\natural}$}
\address{$^\sharp$ Department of Mathematics, 
	      The University of Chicago, 
	      5734 S. University Avenue Chicago, 
	      IL 60637, USA}
\email{souganidis@math.uchicago.edu}

\author[H.V. Tran]{Hung V. Tran$^{\S,*}$}
\address{$^\S$ Department of Mathematics, 
	      The University of Chicago, 
	      5734 S. University Avenue Chicago, 
	      IL 60637, USA}
\email{hung@math.uchicago.edu}
	      
\address{$\natural$ Partially supported by the NSF grants DMS - 0901802 and DMS - 1266383}
\address{$*$ Partially supported by the NSF grant DMS - 1361236}
\title[Large time average of reachable sets]
{Large time average of reachable sets
and Applications to Homogenization of interfaces moving with oscillatory spatio-temporal velocity}

\begin{document}

\begin{abstract}
We study the averaging of fronts moving with positive oscillatory normal velocity, which is periodic in space and stationary ergodic in time. The problem can be reformulated as the homogenization of coercive level set Hamilton-Jacobi equations with spatio-temporal oscillations. To overcome the difficulties due to the oscillations in time and the sublinear growth of the Hamiltonian, we first study the long time averaged behavior of the associated reachable sets using geometric arguments. The results are new for higher than one dimensions even in the space-time periodic setting.
\end{abstract}

\maketitle

\smallskip

\smallskip

\medskip\section{Introduction}\label{sec:Intro}

We investigate  the averaging behavior of fronts moving with positive  oscillatory normal velocity, which is periodic in space and stationary ergodic in time. The problem can be reformulated, using the level-set method,  as the homogenization of coercive level set Hamilton-Jacobi equations with spatio-temporal oscillations. 

In particular, we study the homogenized (averaging) behavior of the solution $u^\ep = u^\ep(x,t,\om)$ to the level-set Hamilton-Jacobi equation
\begin{equation}\label{PDE1}
\begin{cases}
u^\ep_t+a\left(\frac{x}{\ep},\frac{t}{\ep}, \om\right)|Du^\ep|=0 \quad &\text{in} \ \R^n \times (0,\infty),\\
u^\ep =u_0 \quad &\text{on} \ \R^n \times \{0\}, 
\end{cases}
\end{equation}
where $u_0 \in UC(\R^n),$ the space of uniformly continuous functions, and, for each element  $\omega$ of the underlying probability space $(\Om, \cF, \bP)$,  the velocity $a$  is bounded from above and from below away from $0$ uniformly in $x,t$ and $\omega$,  periodic in $x$ and stationary ergodic in $t$; exact definitions are given in the next section. The Hamiltonian $H(x,t,p,\om) = a(x,t,\om)|p|$ inherits the above properties in $x,t$ and, in view of the bounds on $a$,  is coercive and has linear growth in $p$ uniformly in $\omega.$ 

In spite of its simple form and the coercivity of $H$, \eqref{PDE1} falls outside of the scope of the existing homogenization theory for Hamilton-Jacobi equations even in the spatio-temporal periodic setting. The reason is that, contrary to the time independent setting with linear growth and the time dependent setting with superlinear growth, the time oscillations and the linear growth of $H$ in $p$ do not yield any a priori good control, that is independent of $\ep$ estimates,  for the oscillations of the $u^\ep$'s. 

We overcome this difficulty by studying a more general problem, namely  the large time averaging of the reachable sets of the associated control problem which we describe next. 

For $t>0$ and $\omega \in \Omega$, let $\cA_{0,t}(\omega)$ be the set of solutions (admissible paths) to the controlled system 
\begin{equation}\label{control_syst}
\begin{cases}
\gam'(r) = f(\gam(r),r,\xi(r),\omega) := a(\gam(r),r,\omega) \xi(r,\omega) \quad &\text{a.e.} \ r\in (0,t),\\[1mm]
|\xi(r,\omega)| \leq 1 \  \text{a.e.}\   r \in (0,t),%
\end{cases}
\end{equation}
that is 
$$\cA_{0,t}(\omega):=\{\gam:[0,t]\to \R^n\,:\, |\gam'(r)| \leq a(\gam(r),r,\omega) \quad \text{for a.e.} \ r\in [0,t]\}.$$ 


The reachable set $\cR_t(x,\omega)$ at time $t>0$ emanating from $x\in \R^n$ is 
\begin{equation}\label{def:R_t}
\cR_t(x,\omega):=\{y \in \R^n\,:\, \text{there exists}\ \gam \in \cA_{0,t}(\omega) \ \text{such that} \ \gam(0)=x, \ \gam(t)=y\},
\end{equation}
and, here, we are interested in the long time average of  $\cR_t(x,\omega)$, that is the the limit, as $t\to \infty$, of  
\begin{equation}\label{average}
t^{-1} \cR_t(x,\omega).
\end{equation}

%

%

The study of the long time averaged asymptotic behavior of the reachable sets, for spatio-temporal periodic $a$ and $n = 1$, goes back to Poincar\'e and Denjoy; we refer to the book of Arnol'd \cite{Arno88} for details. In this case, for each  $t \geq 0$, $\cR_t(x) = [\gam_L(t),\gam_R(t)]$ with  $\gam_L, \gam_R$  prescribed by the dynamical systems

\begin{equation*}
\begin{cases}
\gam'_R(t)=a(\gam_R(t),t) \ \  \text{for} \ t >0 \ \  \text{and} \ \ \gam_R(0)=x,\\[1mm]
\gam'_L(t)=- a(\gam_L(t),t) \ \  \text{for} \ t >0\  \ \text{and} \ \  \gam_L(0)=x.
\end{cases}
\end{equation*}

The large time average of \eqref{average} is controlled by the limits, as $t \to \infty$, of $t^{-1}\gam_R(t)$ and $t^{-1} \gam_L(t)$, which exist and are the so-called rotation numbers first defined by Poincar\'e. These results were generalized, always when $n=1$,  for $a$ periodic in $x$ and stationary ergodic in $t$ by Li and Lu \cite{ll}.

To the best of our knowledge, the asymptotic behavior of \eqref{average} for $n \geq 2$ was unknown, even in the space-time periodic case. One of our main results, Theorem \ref{thm:STO-R0} below, characterizes this behavior, for spatial periodic and temporal stationary ergodic $a$ and for all dimension $n$.

Next we recall  the relationship  between the reachable sets and the solution of \eqref{PDE1} 
which is based on the control interpretation of  \eqref{PDE1} and is given by the well known Lax-Oleinik formula. 

The Lagrangian $L(x,t,q,\omega)$, that is the Legendre transform of the map $p \mapsto H(x,t,p,\omega)$, is of the form
\begin{equation}\label{L}
L(x,t,q)=
\begin{cases}
0 \ \text{ if} \  |q| \le a(x,t,\omega),\\[1mm]
+\infty \ \text{otherwise}.
\end{cases}
\end{equation}

Hence the action along a path is finite if and only if this path is admissible, in which case the action is actually zero. As a result, the Lax-Oleinik formula formula for solutions to \eqref{PDE1} for $\ep=1$ gives
\begin{equation*}
u (x,t,\omega) = \inf_{y \in \R^n} \{ u_0(y): x\in \cR_t(y,\omega)\},
\end{equation*}
and, after a space-time scaling, 
\begin{equation*}
u^\ep (x,t,\omega) = \inf_{y \in \R^n} \{ u_0(y):  \frac{x}{\ep} \in \cR_{\frac t \ep}(\frac{y}{\ep},\omega)t\}.
\end{equation*}

We show that there exists a convex and compact set $D \subset \R^n$, such that, for every $t>0$, the almost sure  long time average $\left(\frac{t}{\ep}\right)^{-1}\cR_{\frac {t} {\ep}} \left(\frac{y}{\ep}, \omega\right)$, as $\ep \to 0$, is given by $t^{-1}y + D.$  As a result, the set $\{y: \frac{x}{\ep} \in \cR_{\frac t \ep}\left(\frac{y}{\ep},\omega \right)\}$ converges, as $\ep \to 0$,  to the set $\{y: x \in y + tD\}$. It then  follows that $u^\ep(x,t,\omega)$ converges almost surely  to 
$$\overline u(x,t) := \inf_{y \in \R^n}\{u_0(y) \;:\; x \in y + tD\},$$
which is the solution of the homogenized equation 
\begin{equation}\label{takis}
\begin{cases}
\overline u_t + \overline H(D\overline u)=0 \ \text{in} \ \R^n \times (0,\infty),\\[1mm]
\overline u =u_0 \  \text{on} \ \R^n \times \{0\};
\end{cases}
\end{equation}
and the effective Hamiltonian $\ol{H}$ is defined through the set $D$. See Theorem \ref{thm:stoch} and Section \ref{sec:hom} below for the details.

The periodic homogenization of coercive Hamilton-Jacobi equations 
was first studied by Lions, Papanicolaou and Varadhan \cite{LPV} and, later, Evans \cite{Ev:89, Ev:92}. 
Ishii established in \cite{I:00} homogenization in almost periodic settings. The stochastic homogenization of Hamilton-Jacobi equations with convex and coercive Hamiltonians was established independently by Souganidis \cite{S} and Rezakhanlou and Tarver \cite{RT}, and later Schwab \cite{Sch} considered problems with space-time oscillations and superlinear growth in $p$. 
In \cite{LS3} Lions and Souganidis  gave a simpler proof for homogenization in probability using weak convergence techniques. Their program was extended by Armstrong and Souganidis in \cite{AS1,AS2} using the so-called metric problem. 
The homogenization of general nonconvex Hamiltonians in random environments remains to date an open problem. A first extension to level-set convex Hamiltonians was proven by Armstrong and Souganidis in \cite{AS2} and, more recently, Armstrong, Tran and Yu \cite{ATY} established stochastic homogenization for a special class (double-well-type) of  Hamiltonians. 

Few results are available for non-coercive Hamiltonians and they all rely on some reduction property that compensates for the lack of coercivity; see, for example, Alvarez and Bardi \cite{AB:10}, Barles \cite{B:07} and Imbert and  Monneau \cite{IM:08}. A different approach, based on nonresonance conditions, was initiated by Arisawa and Lions \cite{AL:98} and extended to periodic noncoercive-nonconvex Hamiltonians by Cardaliaguet in \cite{C:10}. Of special  interest is the study of noncoercive Hamilton-Jacobi equations associated to moving interfaces. 
The homogenization of  time independent  noncoercive level set equations in the periodic setting was established by Cardaliaguet, Lions and Souganidis \cite{CLS:09} and 
recently by Ciomaga, Souganidis and Tran \cite{CST} in the random setting.  
The homogenization of the G-equation, which is used as model for fronts propagating with normal velocity and advection, in periodic environments was established by Cardaliaguet, Nolen and Souganidis \cite{CNS} (a special case of space periodic incompressible flows was considered by Xin and Yu \cite{XY}) and  by Cardaliaguet and Souganidis in \cite{CS:13} in random media (a special case was studied by Novikov and Nolen \cite{NN}). 

%

We explain now in more detail  the nontrivial difference in the behavior of solutions  when the Hamiltonian $H(x,t,p)$ grows linearly in $p$ and oscillates in time, the setting we study here, and when  one of those two situations does not happen. Indeed, when $H = a(x)|p|$ is coercive, that is $a$ is strictly positive,  and has no time oscillations, then one can easily obtain uniform in $\ep$ and global in time  Lipschitz bounds for $u^\ep$,  which depend only on the Lipschitz constant of  $u_0$ and the bounds on $a$. 
If $H$ grows superlinearly in $p$ and oscillates in time, Cardaliaguet and Silvestre \cite{CS:12} obtained local uniform space-time $C^{0,\alpha}-$ estimates, which only depend on the growth condition of the Hamiltonian but not on its smoothness, for the solutions to the oscillatory Hamilton-Jacobi equations.  In both cases, these estimates are useful and important for the homogenization of \eqref{PDE1}. For the problem at hand, however, no such estimates are known.

Another important difference can be seen in the definition, properties and behavior of  the minimal time function, that is the smallest time it takes to go from one point to other using admissible paths, which is fundamental for the homogenization theory in the absence of time oscillations. 

To explain this we describe next the role of the minimal time function in the study of the long time average of \eqref{average} and, hence,  the homogenization of \eqref{PDE1} when there are no time oscillations.  Throughout this discussion we omit the dependence on $\omega$.


When $a$ and, hence, $f$ in the aforementioned control system are independent of $t$, the reachable set $\cR_t(x)$ is characterized in terms of the minimal time to reach a point $y\in \R^n$ from a point $x\in \R^n$, which is defined  as
\begin{equation}\label{min-f}
\theta(y,x):=\inf\{t\geq 0\,:\,y \in \cR_t(x)\}.
\end{equation}
Indeed, it is immediate  from \eqref{def:R_t} and \eqref{min-f}, that, for all $t>0,$
$$\cR_t(x)=\{y\in \R^n: \theta(y,x) \leq t\}.$$ 
One of the key and very natural property of the minimal time function is subadditivity, that is the fact that, for $x,y,z \in \R^n$, 
$$
\theta(y,x) \leq \theta(y,z)+\theta(z,x).
$$ 
Moreover, In view of the positive lower bound of $a$, there exists a universal constant $C>0$ such that, for all $x,y \in \R^n$, 
\begin{equation}\label{C-lip}
\theta(y,x) \leq C|y-x|.
\end{equation}
Consequently, we can apply the subadditive ergodic theorem to $\theta$ and obtain the large time average of it along any direction. It follows that  there exists a $1$-positively homogeneous, Lipschitz continuous  (with Lipschitz constant $C$ in \eqref{C-lip}), convex $\ol{\theta}:\R^n \to \R$ such that, for any $y\in \R^n$ (and almost surely in $\omega$),
\begin{equation}\label{lim-theta}
\lim_{t\to \infty} \frac{1}{t} \theta(ty,0)= \ol{\theta}(y).
\end{equation}
Using  the relation between minimal time functions and reachable sets, we deduce further that, for all $x\in \R^n$  (and almost surely in $\omega$),
$$
\lim_{t \to \infty} \frac{\cR_t(x)}{t}= \{y\in \R^n\,:\, \ol{\theta}(y) \leq 1\},
$$
and, this, eventually determines the homogenization limit of \eqref{PDE1} when $a$ is time independent.

We discuss now what happens when $a$ and, hence, the control system depends on $t$ and, in the same time, we outline the methodology of the paper.  In this case, it is necessary  to take into account the starting time in the definition of minimal time function $\theta(x,y)$ in \eqref{min-f}  since, for different starting times, the controls for admissible paths are different. 
We thus need to consider the minimal time function in space-time instead of just space variable. 
In this framework, the minimal time to reach a point $(y,t) \in \R^n \times [0,\infty)$ from $(x,0)\in \R^n \times \{0\}$ can be defined as
\begin{equation}\label{min-f-t}
\theta((y,t),(x,0)):=\begin{cases}
t \qquad &\text{if} \ y \in  \cR_t(x),\\
+\infty \qquad &\text{if} \ y \notin \cR_t(x).
\end{cases}
\end{equation}
Let $\alpha>0$ and $\beta>\alpha$ be respectively the lower and upper bounds of $a$.  Then, for all  $t>0$, 
\[
\ol{B}(x,t\al) \subset \cR_t(x) \subset \ol{B}(x,t\beta),
\]
and  the  space-time reachable set $W((x,0)):=\{(y,t) \in \R^n \times [0,\infty)\,:\,\theta((y,t),(x,0))<\infty\}$ emanating from $(x,0)$ 
satisfies 
\begin{equation}\label{cone-like}
\bigcup_{t\geq 0} \ol{B}(x,t \al) \times \{t\} \subset W((x,0)) \subset \bigcup_{t\geq 0} \ol{B}(x,t \beta) \times \{t\}.
\end{equation}
It follows that $W((x,0))$ has a cone-like shape and is controlled by the two lower and upper cones with speed of propagation $\al$ and $\beta$ respectively.
The problem, however, is that $\theta(\cdot,(x,0))=+\infty$ in $\R^n \times [0,\infty) \setminus W((x,0))$. Hence, it does not seem possible to have any estimates like  \eqref{C-lip} or to extend the minimal time function to the whole $\R^n \times [0,\infty)$ as in \cite{CST}.
Because of this lack of control on the minimal time function, we found it necessary  to come up with a new approach.

Here  we look directly at the spatial reachable sets instead of the space-time minimal time function. In some sense, instead of exploring the joint space-time structure of $a$, we let the space periodicity and the time stationarity play different roles in our devising of appropriate subadditive quantities as we explain next. We use $t$ as the index of the subadditive object, which is chosen as some kind of supremum of reachable sets. Indeed in place of $\cR_t(x)$ for each $x$, we consider the enlarged set 
$$\cR_t(Y):=\bigcup_{x\in Y} \cR_t(x),$$ 
with $Y:=[0,1]^n$; at this point we assume that $a$ is $1$-periodic in $x$ for all $t$ and $\omega$.  

It is clear that  $\cR_t(Y)$ serves as a uniform control of $\cR_t(x)$ for $x \in Y$, and, in view of the spatial periodicity of $a$, 
all $x \in \R^n$, since  $\cR_t(x) = [x] + \cR_t(\hat{x})$ where $[x]$ is the integer part of $x$ and $\hat{x} \in Y$. 

As  it should be expected, $\cR_t(Y)$ satisfies some sort of subadditivity property. Indeed, if  $\tilde Y:=-Y$ is the reflected unit cube,  we show in Lemma \ref{lem:STO-subadd} that, for all $m,k \in \N$ with $k\leq m$,
\begin{equation}\label{takis5}
\cR_m(Y) \subset \cR_k(Y)+\cR_{m-k}(Y)+\tilde Y.
\end{equation}

We use then a subadditive ergodic theorem for compact sets due to Sch{\"u}rger \cite{Schurger} and Hansen and Hulse \cite{HH}, which, however, requires convexity; in fact it is explained in  \cite{Schurger} that the result is wrong, in general, for non convex sets. As a consequence this result can not be applied   to the sets $\cR_t(Y)$'s, which are not necessarily convex sets.  Instead we apply the result of \cite{HH} to the convex hull of the $\cR_t(Y)$'s, which are also subbaditive in the sense of \eqref{takis5} and find a compact and convex set $D\subset \R^n$ such that, 
\begin{equation}\label{eq:RY}
\lim_{m \to \infty} \frac{\co \cR_m(Y)}{m}=D,
\end{equation}
which, since, $\cR_m(Y) \subset \co \cR_m(Y)$, is an upper bound for the large time average of $\cR_m(Y)$. The issue is of course to show that $D$ is also a lower bound and, hence, the long time average of $\cR_m(Y)$ itself. As mentioned earlier, this is wrong in general, if we do not have any additional properties of the reachable sets.

In turns out that we can use the structure of the  the control problem to overcome this difficulty and 
this is the key observation in our analysis. In Theorem \ref{thm:STO-R}, we prove that,
for any point $y \in D$, there exists a sequence $\{y_m \in  \frac{\cR_m(Y)}{m} : m \in \N\}$ that stays close to $y$. This is done by designing paths satisfying the ODEs of the controlled system  up to time $m$; the spatial periodicity and temporal stationarity are crucial in this design. It then follows that $D$ is also a lower bound of the large time average of $\cR_m(Y)$ and, hence, it is possible to  remove the convex hull in \eqref{eq:RY}. In view of the monotonicity of $\cR_t(Y)$ in $t$, $D$ is the large time average of $\cR_t(Y)$ along continuous time. Finally, since, in view of the spatial periodicity,  we can control $\cR_t(x)$ by $\cR_t(Y)$ from above and by $\cR_{t-\ell}(Y)$ from below for some finite constant $\ell$, we find that $D$ is also the large time average of $\cR_t(x)$ for all $x \in Y$; see the proofs of Theorem \ref{thm:STO-R} and Theorem \ref{thm:STO-R0} for the details.

The rest of this paper is organized as follows. In the next section, we specify the general assumptions on the velocity $a(x,t,\om)$ and state the main theorems of this paper. In Section \ref{sec:RY} we establish the large time average $D$ of the enlarged reachable sets. We prove the main theorems of this paper in Section \ref{sec:hom}, showing that $D$ is also the large time average of the reachable set starting from any point in the unit cube uniformly, and we apply this result to homogenize \eqref{PDE1}. Finally in Section \ref{sec:applications}, we study the homogenization of moving fronts where there is an ambient drift in the velocity using the asymptotic behavior of the corresponding reachable set; we also investigate the homogenization of a non-coercive Hamilton-Jacobi equation.

\subsection*{Notations}
We work in the $n$-dimensional Euclidean space $\R^n$ and we denote by $\Z^n$ the set of points with integer coordinates. $\N$ denotes the set of natural numbers including zero. Let $Y$ be the unit cell $[0,1]^n$ and and $\tilde Y:=-Y=[-1,0]^n$. For any $x \in \R^n$, we set $([x],\hat{x})$ to be the unique pair in $\Z^n \times [0,1)^n$ such that $x = [x] + \hat{x}$. The open ball in $\R^n$ centered at $x$ with radius $r > 0$ is denoted by $B_r(x)$, and this notation is further simplified to $B_r$ if the center is the origin. The cardinality of a set $K$ that has finite number of elements is denoted by $\text{Card} (K)$. The set of non-empty compact subsets of $\R^n$ is denoted by $\srC$. For any $A,B \in \srC$ and any $c \in \R$, we set $A+B := \{x+y \,:\, x \in A, y \in B\}$ and $cA := \{cx \,:\, x \in A\}$. The Hausdorff metric $\rho$ on $\srC$ is defined as $\rho(A,B) := \max\{\sup_{x \in A} \inf_{y \in B} |x-y|, \sup_{x \in B} \inf_{y \in A} |x-y|\}$. For any $A \in \srC$, $\|A\|:=\max \{|x| \,:\, x \in A\}$ and  $\co A$ and $\mathcal{E}(A)$ are respectively the convex hull and set of extreme points of $A$. The set of non-empty compact and convex subsets of $\R^n$, which is a closed subset of $\srC$,  is denoted by $\co \srC$; $C^{0,1}(\R^{n+1})$ s the set of bounded, Lipschitz continuous defined on $\R^n$, $\|\cdot \|_\infty$ is the $L^\infty$-norm of a bounded function,  and $\mathcal B(\Xi)$ is the Borel $\sigma$-algebra the metric space $\Xi$. 

\section{Assumptions, preliminaries and main results}
\label{sec:prelim}

\subsection*{The setting and assumptions}

We consider a probability space $(\Omega, \cF, \bP)$ endowed with an {ergodic group of measure preserving transformations} $(\tau_k)_{k\in\Z}$, that is, a family of maps $\tau_k:\Omega\to\Omega$ satisfying,  for all $k,k'\in \Z$ and all $\cU\in\cF$,
	\[\tau_{k+k'}=\tau_{k}\circ\tau_{k'}\;\; \hbox{ and }\;\;
	\bP[\tau_k\cU] = \bP[\cU]\]
and 
	\[\hbox{ if }\tau_k(\cU)=\cU \hbox{ for every } k\in\Z,
	\hbox{ then either }\bP[\cU] = 1\hbox{ or }\bP[\cU]=0.\]

As far as  $a:\R^n \times \R \times\Omega\to\R$ is concerned, we assume that 
\begin{enumerate}
\item[(A0)]  $a$ is measurable with respect to $\cB(\R^{n+1})\times\cF$,
 \item[(A1)]  $a$ is $\Z^n$-periodic in $x$ and stationary in $t$ with respect to $(\tau_k)_{k\in\Z}$, that is, for every $(x,l)\in \R^n \times \Z^n$, $(t,k)\in \R \times \Z$, and $\omega\in\Omega$, 
\[
a(x+l,t,\tau_k\omega) = a(x,t+k,\omega),
\]
 
 \item[(A2)] $a(\cdot,\cdot,\omega) \in C^{0,1}(\R^{n+1})$ for each $\omega$ and there exist $\al,\beta>0$ such that, for all $(x,t)\in \R^{n+1}$ and $\omega \in \Omega$,
\begin{equation}\label{A2}
\al \leq a(x,t,\omega) \leq \beta.
\end{equation}
\end{enumerate}

For simplicity, we combine all the assumptions into
\begin{enumerate}
\item[(A)] $a = a(x,t,\om)$ satisfies (A0), (A1) and (A2).
\end{enumerate}

\subsection*{The admissible paths and reachable sets} We recall and define some notions concerning the reachable sets.
For $t\ge s$ and $\omega\in \Omega$, the set $\cA_{s,t}(\omega)$ of admissible paths is given by 
\begin{equation}\label{def:A-Om}
\cA_{s,t}(\omega) := \{\gam:[s,t] \to \R^n \,:\, |\gam'(r)| \leq a(\gam(r),r,\omega) \quad \text{for a.e.} \ r\in [s,t]\}.
\end{equation}
The space-time reachable set corresponding to $(x,s,\omega)\in \R^n\times \R\times \Omega$ is defined by
\begin{equation}\label{def:Gam-Om}
\Gam(x,s)(\omega):=\{(y,t)\in \R^{n}\times [s,\infty):  \text{there exists} \ \gam \in \cA_{s,t}(\omega) \ \text{such that}\ \gam(s)=x,\  \gam(t)=y\}.
\end{equation}
For $t\geq s$, the (space-time) reachable set from $(x,s)$ at time $t$ is
\begin{equation}\label{def:Gam_t-Om}
\Gam_t(x,s)(\omega) := \Gam(x,s)(\om) \cap  (\R^n \times [s,t]). 
\end{equation}
The projection of $\Gam_t(x,s)(\om)$ on $\R^n$ is 
given by
\begin{equation}\label{def:R-Om}
\cR_t(x,s)(\omega): = \{y \in \R^n \,:\, \text{there exists}\ \gam \in \cA_{s,t}(\om) \ \text{such that}\ \gam(s)=x,\ \gam(t)=y \};
\end{equation}
note that, in view of the discussion in the Introduction, $\cR_t(x,s)(\omega)$  is the (spatial) reachable set at time $t$ starting from $x$ with initial time $s$. 

In our analysis, we use  the ``enlarged'' reachable set 
\[
\cR_t(Y,s)(\omega):= \cup_{x\in Y} \cR_t(x,s)(\omega),
\]
which is the set of spatial points reachable at $t$ starting from the unit cell $Y$ at time $s$. 

Throughout the paper, when  the initial time $s=0$,  we write 
$$\cR_t(x)(\omega):=\cR_t(x,0)(\omega) \ \text{and}  \  \cR_t(Y)(\om) := \cR_t(Y,0)(\om).$$

We discuss next some properties of the reachable sets which we will use in the sequel. 

The reachable set $\cR_t(x,s)$ for all $x \in \R^n$, $s \in \R$ and $t \ge s$, is compact, and so is the enlarged reachable set $\cR_t(Y,s)$; see for instance Cannarsa and Frankowska \cite{CF} and the references therein. If the control system \eqref{control_syst} is linear, that is if $f(x,t,\xi)=Mx+L\xi$, then $\cR_t(x)$ is convex. In general,  when $f$ is nonlinear in $x$ and $\xi$, then $\cR_t(x)$ is not convex. Since we do not assume that $a$ is linear, $\cR_t(x)$ is presumably not convex, which makes the study of large time limit of $\cR_t(x)$ much more interesting. 

The next observation was already discussed in the Introduction.

\begin{lem}\label{lem:Rab} 
Assume {\upshape(A)}.  Then, for any $(x,t,\omega) \in \R^n \times [0,\infty) \times \Omega$, 
\begin{equation}\label{Rt-base}
\ol{B}_{t\al}(x) \subset \cR_t(x)(\omega) \subset \ol{B}_{t\beta}(x) 
\end{equation}
and
\begin{equation}\label{Rt-Y-base}
\ol{B}_{t\al}(0) \subset \cR_t(Y)(\om) \subset \ol{B}_{t\beta}(0) + Y.
\end{equation}
\end{lem}

\begin{proof}
If $|y-x| \le \alpha t$, then the straight line path connecting $y$ to $x$ with speed $\alpha$ is admissible, which implies the first inclusion of \eqref{Rt-base}. On the other hand, for any $\gam \in \cA_{0,t}(\om)$ with $\gam(0) = x$, $|\gam(t) - x| \le \|a\|_\infty  t \le \beta t$, which yields the second inclusion of \eqref{Rt-base}.

In view of $\cR_t(0)(\om) \subset \cR_t(Y)(\om)$, the first inclusion of \eqref{Rt-Y-base} follows from that of \eqref{Rt-base}. On the other hand, from the definition of $\cR_t(Y)(\om)$ and the second inclusion in \eqref{Rt-base}, we have $\cR_t(Y)(\om) \subset \bigcup_{x \in Y} \ol{B}_{t\beta}(x) \subset \ol{B}_{t\beta}(0) + Y$.
\end{proof}

The second result implies  that the reachable sets grow, a fact that follows from their monotonicity in time.

\begin{lem}\label{lem:monot} Assume {\upshape(A)} and fix $\om \in \Om$ and $x \in \R^n$. For any $s \in \R$ and $t_2 \ge t_1 \ge s$, 
\begin{equation}\label{mono-t}
\cR_{t_1}(x,s)(\om) \subset \cR_{t_2}(x,s)(\om),
\end{equation}
and, for any $t \in \R$ and $s_1 \le s_2 \le t$, 
\begin{equation}\label{mono-s}
\cR_{t}(x,s_2)(\om) \subset \cR_{t}(x,s_1)(\om).
\end{equation}
These relations are still true  for the enlarged reachable sets.
\end{lem}

\begin{proof}
Fix  $y \in \cR_{t_1}(x,s)$, choose $\gam \in \cA_{s,t_1}(\om)$ such that $\gam(0) = x$ and $\gam(t_1) = y$, and define $\widetilde\gam:[0, t_2] \in \R^n$ by $\widetilde\gam(r) = \gam(r)$ for $r \in [0,t_1]$ and $\widetilde\gam(r) = y$ for $r \in [t_1,t_2]$. It is immediate that  $\widetilde\gam \in \cA_{s,t_2}(\om)$ and $\widetilde\gam(t_2) = y$. This proves \eqref{mono-t}. The inclusion  \eqref{mono-s} and the corresponding results for $\cR_t(Y,s)$ are proved similarly.
\end{proof}

\subsection*{The subadditive ergodic theorem for compact convex sets}


A key tool that we will use is the subadditive ergodic theorem for compact convex sets. We say that a family of $\srC$-valued random sets $X = \left(X_{k,m}(\om)\right)_{0 \le k < m}$, where $k,m \in \N$, is stationary if
\begin{equation*}
X_{m+l,k+l}(\om) = X_{m,k}(\tau_l \om), \quad  \text{for all} \quad l, m, k\in \N, \ m \le k, \ \text{and}\ \om \in \Om,
\end{equation*}
and subadditive if
\begin{equation*}
X_{m,k}(\om) \subset X_{m,l}(\om) + X_{l,k}(\om), \quad \text{for all} \quad l,m,k \in \N , \ m < l < k \ \text{and} \ \om \in \Om.
\end{equation*}

A more general version of the next result is proved in \cite{Schurger} and \cite{HH}

\begin{thm}[Subadditive ergodic theorem]\label{thm:subadd} Let $X = \left(X_{m,k}(\om)\right)_{0\le k < m}$ be a stationary subadditive family of $\co \srC$-valued random sets defined on $(\Om,\cF,\bP)$ and assume that $\E \|X_{0,1}\| \le C$ for some $C>0$. Then there exists a $\co \srC$-valued set $D$ and a subset $\Om_1 \subset \Om$ with full measure, such that $n^{-1}X_{0,n}(\om)$ converges to $D$ in $(\srC,\rho)$, as $n \to \infty$, for all $\om \in \Om_1$.
\end{thm}

\subsection*{Main theorems}

The first main theorem of this paper concerns the large time average of the reachable sets starting from any point in the unit cell.

\begin{thm}\label{thm:STO-R0}
Assume {\upshape(A)}. There exists a compact and convex  $D\subset \R^n$ and an event $\widetilde \Omega \in \cF$ of full probability such that, for each $\om \in \widetilde\Omega$ and any $x\in Y$, 
\begin{equation}\label{lim:R0t}
\lim_{t\to \infty}\frac{\cR_t(x)(\om)}{t}=D \quad \text{in} \quad (\srC,\rho)
\end{equation}
and,
\begin{equation}\label{lim:Rt-uni}
\lim_{t\to \infty} \sup_{x\in Y} \rho\left(\frac{\cR_t(x)(\om)}{t},D\right)=0.
\end{equation}
\end{thm}


Next, we identify the effective Hamiltonian from the compact convex set $D$ of Theorem \ref{thm:STO-R0} as follows. Let
\[
\ol{L}(q):=\begin{cases}
0 \qquad &\text{for} \ q \in D,\\
+\infty \qquad &\text{otherwise,}
\end{cases}
\]
and, for $p\in \R^n$, define
\begin{equation}\label{Hbar-def}
\ol{H}(p):=\sup_{q\in \R^n} \left(p\cdot q - \ol{L}(q)\right)=\sup_{q\in D} \ p\cdot q.
\end{equation}
It is straightforward that $\ol{H}$ is convex and $1$-positively homogeneous. 

Let $\overline u$ be the solution of the following equation
\begin{equation}\label{PDE-hom}
\begin{cases}
\overline u_t+\ol{H}(D\overline u)=0 \quad &\text{in} \ \R^n \times (0,\infty),\\[1mm]
\overline u =u_0 \quad &\text{on} \ \R^n \times \{0\}.
\end{cases}
\end{equation}

The homogenization result is:

\begin{thm}\label{thm:stoch}
Assume {\upshape (A)} and let $\widetilde\Omega$ be as defined in Theorem \ref{thm:STO-R0}. Then, for each $\omega\in\widetilde\Omega$, 
the solution $u^\ep=u^\ep(\cdot, \cdot, \omega)$ of \eqref{PDE1} converges locally uniformly in 
$\R^n \times [0,\infty)$ to the solution $\overline u $ of \eqref{PDE-hom}.
\end{thm}

\section{Large time average of the enlarged reachable sets}
\label{sec:RY}

\subsection*{Some properties of the reachable sets and the admissible paths} We investigate, using  {\upshape(A 0), 
the behavior of the the reachable set $\cR_s(x,t)(\om)$, when $x$ and $t$ are translated, as well as its subadditivity properties.

\begin{lem}\label{lem:R_rand} Assume {\upshape(A)}. For any $x \in \R^n$, $t \ge 0$, $k \in \N$ and $\omega \in \Omega$,
\begin{equation}\label{Rx_rand}
\cR_t(x)(\omega) = [x] + \cR_t(\hat{x})(\omega)
\end{equation}
and
\begin{equation}\label{Rt_rand}
\cR_{k+t}(x,k)(\omega) = \cR_t(x)(\tau_k \omega).
\end{equation}
\end{lem}
\begin{proof}  For any  $y \in \cR_t(x)(\om)$, choose $\gam \in \cA_{0,t}(\om)$ satisfying $\gam(0)=x$ and $\gam(t) = y$, and define $\tilde{\gam} : [0,t] \to \R^n$ by $\tilde{\gam}(\cdot) = \gam(\cdot) - [x]$. The periodicity in space of $a$ yields $\tilde{\gam} \in \cA_{0,t}(\om)$. Moreover, $\tilde{\gam}(0) = \hat{x}$, $\tilde{\gam}(t) = y - [x]$ and, hence,  $y-[x] \in \cR_t(\hat{x})(\om)$ and thus \eqref{Rx_rand} follows. 
The other direction of inclusion in \eqref{Rx_rand} follows similarly.

To prove \eqref{Rt_rand}, for any $y \in \cR_{k+t}(x,k)(\omega)$, choose $\gam \in \cA_{k,k+t}(\omega)$ such that $\gam(k) = x$ and $\gam(k+t) = y$ and define $\tilde{\gam} : [0,t] \to \R^n$ by $\tilde \gam(\cdot) = \gam(\cdot + k)$. Then, f $\tilde{\gam}(0) = x$ and $\tilde{\gam}(t) = y$ and, for a.e.  $r \in (0,t)$, 
$$
|\tilde{\gam}'(r)| = |\gam'(k+r)| \le a(\gam(k+r),k+r,\omega) = a(\tilde{\gam}(r),r,\tau_k \omega).
$$
It follows that $\tilde{\gam} \in \cA_{0,t}(\tau_k \omega)$,  $y \in \cR_t(x)(\tau_k \omega)$, and, hence,  $\cR_{k+t}(x,k)(\omega) \subset \cR_t(x)(\tau_k \omega)$. The other inclusion follows in  the same way.
\end{proof}

\begin{lem}\label{lem:STO-subadd}
Assume {\upshape(A)}. Then, for any $t \in \R$, $s \in \N$ such that $t \geq s$ and  $\om \in \Om$,
\begin{equation*}
\cR_t(Y)(\omega) \subset \cR_s(Y)(\omega)+\cR_{t-s}(Y)(\tau_s \omega)+\tilde Y.
\end{equation*}
\end{lem}
\begin{proof}
For  $y\in \cR_t(Y)(\omega)$, choose $\gam \in \cA_{0,t}(\omega)$ such that $\gam(0)\in Y$ and $\gam(t)=y$. Then, in light of \eqref{Rx_rand} and \eqref{Rt_rand}, we get
\begin{align*}
y=\gam(t)&=\gam(s)+(\gam(t)-[\gam(s)])+([\gam(s)]-\gam(s))\\[1mm]
 &\in \cR_s(Y)(\omega)+\left(\cR_{t}(\gam(s),s)(\omega)-[\gam(s)]\right)+\tilde Y\\[1mm]
 &= \cR_s(Y)(\omega)+\left(\cR_{t-s}(\gam(s))(\tau_s \omega) -[\gam(s)]\right) +\tilde Y\\[1mm]
 &= \cR_s(Y)(\omega)+\cR_{t-s}(\widehat{\gam(s)})(\tau_s \omega) +\tilde Y\\[1mm]
 &\subset \cR_s(Y)(\omega)+\cR_{t-s}(Y)(\tau_s \omega)+\tilde Y,
\end{align*}
which is the desired conclusion.
\end{proof}

Since taking the closed convex hull is a linear operation and $\tilde Y$ is convex and compact, it follows from Lemma \ref{lem:STO-subadd} that
\begin{equation}\label{coR_subadd}
\co \cR_m(Y)(\om) \subset \co \cR_k(Y)(\om) + \co \cR_{m-k}(Y)(\tau_k \om)+\tilde Y.
\end{equation}

In the analysis that follows we will need to construct admissible curves connecting any two points within a uniform time.  This is the topic of the following lemma. 

\begin{lem}\label{lem:ODE} Assume {\upshape(A)} and let $\ell : = [\sqrt{n}/\alpha] + 1$. For each $\om \in \Om$ and any $y_1, y_2 \in Y$, there exists $\gam_{y_1,y_2} \in \cA_{0,\ell}(\om)$ such that $\gam_{y_1,y_2}(0) = y_1$ and $\gam_{y_1,y_2}(\ell) = y_2$.
\end{lem}

\begin{proof} If $y_1 = y_2$, the path $\gam_{y_1,y_2}(\cdot) \equiv y_1$ yields the desired result. Hence, we assume $y_1 \ne y_2$ and note that $|y_2 - y_1| \le \sqrt{n}$ and thus $|y_2-y_1|/\alpha \le \ell$.  It is immediate that the path satisfies the claim
\begin{equation}
\gam_{y_1,y_2}(t) = \begin{cases}
\displaystyle y_1 +  \frac{t\alpha (y_2 - y_1)}{|y_2 - y_1|}, \quad & t \in [0, |y_2-y_1|/\alpha],\\
y_2, \quad & t \in (|y_2-y_1|/\alpha, \ell],
\end{cases}
\end{equation}
satisfies the claim.
\end{proof}

\subsection*{The large time average of the enlarged reachable set}

We prove the following theorem which identifies the large time average of the enlarged reachable set.
\begin{thm}\label{thm:STO-R} Assume {\upshape(A)}. There exist a compact and convex set $D\subset \R^n$ and an event $\widetilde\Om \in \cF$ of full probability such that, for any $\om \in \widetilde\Omega$,
\begin{equation}
\quad \lim_{t\to \infty} \frac{\cR_t(Y)(\omega)}{t} = D  \quad \text{in} \quad (\srC,\rho).
\label{lim:STO-R}
\end{equation}
\end{thm}

Following the strategy outlined in the Introduction, we first need identify the long time behavior of the convex hull $\co \cR_m(Y)(\om)$ of the reachable set $\cR_m(Y)(\om)$, along integer time $m \in \N$. To this end,  consider the family  $X := \left(X_{m,k}(\om)\right)_{0\le m < k}  \subset \co \srC$ given by 
\begin{equation}\label{def:Xmk}
X_{m,k}(\om) := \co \cR_{k-m}(Y)(\tau_m \om) + \tilde{Y} = \co (\cR_{k-m}(Y)(\tau_m \om) + \tilde{Y}).
\end{equation}

\begin{thm}\label{thm:XR} 
Assume {\upshape(A)}. There exist a compact and convex set $D\subset \R^n$ and an event of full probability $\Omega_0\subseteq\Omega$ such that, for any $\om \in \Omega_0$,
\begin{equation}
\lim_{m\to \infty} \frac{X_{0,m}(\om)}{m} = D \qquad \text{in} \quad (\srC,\rho),\label{lim:X}
\end{equation}
and
\begin{equation}
\lim_{m\to \infty} \frac{\co \cR_m(Y)(\om)}{m} = D  \qquad \text{in} \quad (\srC,\rho).\label{lim:coR}
\end{equation}
Moreover, 
\begin{equation}\label{Dab}
\ol{B}_\al \subset D \subset \ol{B}_\beta.
\end{equation}
\end{thm}
\begin{proof} It is immediate from \eqref{def:Xmk} and 
\eqref{coR_subadd} that  shows that $X$ is stationary and  subadditive. Moreover, in view of Lemma \ref{lem:Rab}, $\E \|X_{0,1}\|$ is finite. It follows from  Theorem \ref{thm:subadd} that there exist a convex set $D \in \srC$ and a set of full probability $\Om_0 \subseteq \Om$ such  that, for every $\om \in \Om_0$, \eqref{lim:X} holds. Upon redefining $\Om_0$ as $\bigcap_{j \in \Z} \tau_j \Om_0$, we may assume that $\Om_0$ is invariant under integral translations and has full measure. Note that \eqref{lim:X} holds for all $\om \in \Om_0$.

To prove \eqref{lim:coR} we fix $\om \in \Omega_0$ and  for simplicity,  we omit the dependence of $\cR_m(Y)$ on it. Since
\begin{equation*}
\rho \left(\frac{\co \cR_m(Y)}{m},D\right) \le \rho \left(\frac{\co \cR_m(Y)}{m}, \frac{\co \cR_m(Y) + \tilde{Y}}{m}\right) + \rho \left(\frac{\co \cR_m(Y) + \tilde{Y}}{m},D\right),
\end{equation*}
it is enough to show that the right hand side of the above inequality converges to $0$ as $m\to \infty$. Indeed,  in view of \eqref{lim:X}, the second term on the right hand side above approaches zero as $m \to \infty$, while,  the first term, in light of 
\begin{equation}
\rho(A,A+B) \le \|B\| \quad \text{for all} \quad A, B \in \srC,
\end{equation}
is bounded by $\|\tilde{Y}/m\|$ and, hence, also  converges to zero.

The last claim is immediate from  Lemma~\ref{lem:Rab} and  \eqref{lim:X} and \eqref{lim:coR}.     
\end{proof}


In the following lemma and for future use we show that almost surely the convergence in \eqref{lim:coR} holds simultaneously  for a special family of translations of the realization. The proof  is technical, but the benefit of this lemma will be clear later. We recall that $\ell = [\sqrt{n}/\alpha] + 1$.

\begin{lem}\label{lem:STO-coR}
Assume {\upshape(A)}. Let $D$ be as in Theorem~\ref{thm:XR}. There exists an event $\widetilde\Om \subset \Om$ of full probability measure, such that for each $\omega\in\widetilde\Omega$ and any integer $s \geq 0$,
\begin{equation}
\lim_{m\to \infty} \frac{\co \cR_m(Y)(\tau_{s(m+\ell)}\om)}{m} = D \quad \text{in} \quad (\srC,\rho).\label{lim:STO-coR_trans}
\end{equation}
\end{lem}

\begin{proof}
For each fixed $s \in \N$, we construct $\Om_s \in \cF$ with $\bP(\Om_s) = 1$ such that \eqref{lim:STO-coR_trans} holds for $s$. The conclusion then follows  once we define $\widetilde\Om := \bigcap_{s\in \N} \Om_s$.

For $s = 0$, let $\Om_0$ be as defined in Theorem \ref{thm:XR} and \eqref{lim:STO-coR_trans} follows. It remains to define $\Om_s$ for any fixed $s \ge 1$. In view of the pointwise convergence in $\Om_0$ and Egoroff's theorem, for any $\delta \in (0,1)$, there exists $W_{s,\delta} \subset \Om_1$ with $\bP(W_{s,\delta}) > 1-\frac{\delta}{4s\beta}$ and $M_{s,\delta} \in \N$ such that, if $m \ge M_{s,\delta}$, then
\begin{equation*}
\sup_{\om \in W_{s,\delta}} \; \rho \left(\frac{\co \cR_m(Y)(\tau_{s\ell}\om)}{m}, D\right) < \frac{\delta}{4}.
\end{equation*}

Applying the ergodic theorem to the indicator function ${\bf 1}_{W_{s,\delta}}$, we find  $\Om_{s,\delta} \in \cF$ with $\bP(\Om_{s,\delta}) = 1$ so that, for each $\om \in \Om_{s,\delta}$,
\begin{equation}
\lim_{N \to \infty} \frac{1}{N+1} \sum_{k=0}^{N} {\bf 1}_{W_{s,\delta}}(\tau_k \om)  \; = \; \bP(W_{s,\delta}) > 1 - \frac{\delta}{4s\beta}.
\label{ergodic_cR}
\end{equation}

Let $\Om_s := \bigcap_{\delta \, \in \, \bQ \cap (0,1)} \Om_{s,\delta}$.
It is clear that $\Om_s \in \Om_1$ and $\bP(\Om_s) = 1$. We need to check that \eqref{lim:STO-coR_trans} holds for all $\om \in \Om_s$. Fix any such $\om \in \Om_s$ and observe that for $\ep > 0$, in view of \eqref{ergodic_cR},
 there exists $N_{s,\ep} > 0$ so that, for $N \ge N_{s,\ep}$,
\begin{equation}
\sum_{k=0}^{N} {\bf 1}_{W_{s,\ep}}(\tau_k \om)  = \text{Card}\left\{k \in [0,N] \; : \: k \in \N, \tau_k \om \in W_{s,\ep} \right\}  > \left(1 - \frac{\ep}{4s\beta} \right) (N+1).
\label{ergodic_cR1}
\end{equation}

Set $M_{s,\ep} := \max\{N_{s,\ep}/(2s), 8\sqrt{n}/\ep \}$. Then, for any $m \ge M_{s,\ep}$, the previous claim with $N = 2sm$ yields that, inside the set $\N_{\le 2sm} := \{0,1,\cdots,2sm\}$, there is no subset $\cJ = \{k, k+1, \cdots, k+[\frac{\ep m}{2\beta}]\}$, which consists of $[\frac{\ep m}{2\beta}] + 1$ consecutive integers such  that, if $j \in \cJ$, then $\tau_j \om$ fails to be in $W_{s,\ep}$. 

Consequently, there exists an integer $\tilde r \le sm$ with $sm - \tilde r \leq [\frac{\ep m}{2\beta}]$ such that $\tau_{\tilde r} \om \in W_{s,\ep}$. 

Next observe that
\begin{equation*}
\begin{aligned}
\rho \left(\frac{\co \cR_m(Y)(\tau_{s(m+\ell)} \om)}{m}, D\right) \le &\ \rho \left(\frac{\co \cR_m(Y)(\tau_{s(m+\ell)} \om)}{m}, \frac{\co \cR_m(Y)(\tau_{\tilde r}\circ \tau_{s\ell} \om)}{m}\right) \\
& \quad + \rho \left(\frac{\co \cR_m(Y)(\tau_{\tilde r}\circ \tau_{s\ell} \om)}{m}, D\right).
\end{aligned}
\end{equation*}

Since $\tau_{\tilde r}\om \in W_{s,\ep}$, the second term in the right hand side of the inequality above is bounded from above by $\ep/4$. To estimate the first term, we note that $0<  s(m+\ell) -\tilde{r} + s\ell \leq [\frac{\ep m}{2\beta}]$. 

In view of Lemma \ref{lem:STO-subadd}, we have
\begin{equation*}
\begin{aligned}
\cR_m(Y)(\tau_{\tilde r+s\ell} \om) \subset &\  \cR_{sm-\tilde r}(Y)(\tau_{\tilde r+s\ell}\om) + \cR_{m - ({sm-\tilde r})}(Y)(\tau_{s(m+\ell)}\om) + \tilde{Y}\\
\subset &\ \cR_{sm-\tilde r}(Y)(\tau_{\tilde r+s\ell}\om) + \cR_{m}(Y)(\tau_{s(m+\ell)}\om) + \tilde{Y}.
\end{aligned}
\end{equation*}

On the other hand, the stationarity 
yields
$$
\cR_m(Y)(\tau_{s(m+\ell)} \om) = \cR_{m + sm - \tilde r} (Y,sm-\tilde r)(\tau_{\tilde r+s\ell} \om),
$$

Next we apply Lemma \ref{lem:STO-subadd} to the set on the right,
and get
\begin{equation*}
\begin{aligned}
\cR_m(Y)(\tau_{s(m+\ell)} \om) \subset &\ \cR_{sm-\tilde r}(Y, sm - \tilde r)(\tau_{\tilde r+s\ell}\om) + \cR_{m}(Y, sm - \tilde r)(\tau_{s(m+\ell)}\om) + \tilde{Y}\\
\subset &\ \cR_{sm-\tilde r}(Y)(\tau_{\tilde r+s\ell}\om) + \cR_{m}(Y)(\tau_{s(m+\ell)}\om) + \tilde{Y},
\end{aligned}
\end{equation*}
where the second line follows from Lemma \ref{lem:monot}.

These relations, together with the fact that $\cR_{sm-\tilde r}(Y) \subset \overline{B}_{[\frac{\ep m}{2\beta}]\beta} + Y$, imply that
\begin{equation*}
\rho \left(\frac{\co \cR_m(Y)(\tau_{s(m+\ell)} \om)}{m}, \frac{\co \cR_m(Y)(\tau_{\tilde r}\circ \tau_{s\ell} \om)}{m}\right) \le \frac{1}{m}\left(\frac{\ep m}{2} + 2\sqrt{n}\right) \le \frac{3\ep}{4}.
\end{equation*}

Combining this with the previous estimate, we showed that \eqref{lim:STO-coR_trans} holds for the fixed $s$ and any $\om \in \Om_s$. This verifies the eligibility of $\Om_s$ and the proof of the lemma is complete.
\end{proof}

It follows from Theorem~\ref{thm:XR} that the average of the convex hull of $\cR_m(Y)(\om)$ converges to $D \in \co \srC$, which is an upper bound of $\lim_{m \to \infty} \cR_m(Y)(\om)/m$ should the latter exist. To show that they are equal, it remains to prove that
$$ \lim_{m\to \infty} \sup_{x \in D} d(x, \cR_m(Y)(\om)/m) =0,$$ 
and, hence, in view of the the compactness of $D$, it suffices to show  that
$$\lim_{m \to \infty} d(x,\cR_m(Y)(\om)/m) = 0 \  \text{ for any fixed $x \in D$}.$$

This last limit  is the key difficulty in the whole proof. 
As a first step we use the convexity of $D$ and some basic convex analysis to prove the convergence result for any $y \in \cE(D)$, the set of extreme points of $D$. 

We recall that  $e \in D$ is an extreme point of a compact and convex set $D \subset \R^n$ if, whenever  
$e = \lambda x + (1-\lambda)y$ with  $x,y \in D$ and $\lambda \in [0,1]$, then either  $x = e$ or $y=e$.  Moreover  $p \in D$ is exposed, if there exists a linear functional $f: \R^d \to \R$ such that $f(p) > f(p')$ for all $p' \in D \setminus \{p\}$. 

\begin{lem}\label{lem:STO-ExtrR} Assume {\upshape(A)} and  let $D$ and $\widetilde\Omega$  be as in Theorem~\ref{thm:XR} and  Lemma~\ref{lem:STO-coR} respectively. For each extreme point $y$ of $D$, $\om \in \widetilde\Omega$ and $s \in \N$, 
\begin{equation}
 \quad \lim_{m\to \infty} d\left(y, \; \frac{\cR_m(Y)(\tau_{s(m+\ell)}\om)}{m}\right) = 0. \label{lim:STO-ExtrR}\\
\end{equation}
\end{lem}
\begin{proof} Since, in view of the Straszewicz's theorem \cite[Theorem 18.6]{Rockafellar}, every extreme point of $D$ is the limit of some sequence of exposed points of $D$, without loss of generality, we assume that $y$ is an exposed point of $D$ and choose a linear function $f: \R^n \to \R$ such that $f(y) > f(x)$ for any $x \in D\setminus\{y\}$. 

For each $m$, assume that $f$ achieve its maximum in $\co \cR_m(Y)(\tau_{s(m+\ell)}\om)$ at $x_m$.
Without loss of generality, we may assume that $x_m \in \mathcal{E}(\co \cR_m(Y)(\tau_{s(m+\ell)}\om))$. 
Then in view of \eqref{lim:STO-coR_trans},
$$
\lim_{m \to \infty} f(x_m/m) = f(y).
$$
It follows that every cluster point of $x_m/m$, which is in $D$, coincides with $y$. This shows that
\begin{equation*}
\lim_{m \to \infty} d\left(y, \; \mathcal{E}\left(\frac{\co \cR_m(Y)(\tau_{s(m+\ell)}\om)}{m}\right)\right) = 0.
\end{equation*}
The desired  limit follows from the fact that $\mathcal{E}(\co A) \subset A$ for any compact set $A \subset \R^n$.
\end{proof}

With all the previous facts at hand, we may now proceed to the proof of Theorem \ref{thm:STO-R}. As mentioned earlier, the claim  will follow if we show $\lim_{k \to \infty} d(x,\cR_k(Y)(\om)/k) = 0$ for an arbitrary $x \in D$. That is, for any $\ep$ neighborhood $V$ of $x$ and $k$ sufficiently large, we need to find $\gam \in \cA_{0,k}(\om)$ such that $\gam(k)/k \in V$.   Next we explain briefly the idea of how to construct $\gam$.

We use the Minkowski-Carath\'eodory theorem \cite[Theorem 8.11]{Simon}  to express $x$ as a convex combination  of $n+1$ extreme points $(y_i \in \cE(D))_{i=1,\ldots,n+1}$. The reason for this is that, for each extreme point $y_i$, it is possible to  find sub-paths $\gam_{ij} \in \cA_{0,m}(\tau_{s_j(m+\ell)}\om)$, for  some  sequence $(s_j)_j$ and and appropriately chosen  large  $m\in \N$, such that $\gam_{ij}(m)/m$ lies in a small neighborhood of $y_i$. Then we  construct the desired $\gam \in \cA_{0,k}(\om)$ by translating and connecting those sub-paths $(\gam_{ij})_{ij}$. In the proof that follows, we carefully carry out these arguments. It turns out that, the integer $m$ and the sequence $(s_j)_j$ can be chosen according to some rational approximation of the coefficients in the convex combination of $x$. The periodicity in space and the stationarity in time are crucial for this construction to work.

\begin{proof}[Proof of Theorem \ref{thm:STO-R}]  Since the argument is long, we divide the proof in three steps.

{\itshape Step 1: Pointwise convergence in Euclidean distance.} Let $\widetilde\Om$ be as in Lemma~\ref{lem:STO-coR}. We show that,  for each fixed $x\in D$ and $\om \in \widetilde \Om$,
\begin{equation}\label{STO-s1}
\lim_{k\to \infty} d\left(x, \frac{\cR_k(Y)(\om)}{k}\right)=0.
\end{equation}

As explained in the discussion prior to the proof, 
there exist $n+1$ extreme points $y_1, \ldots, y_{n+1}$ of $D$ and $n+1$ numbers $\lambda_1, \ldots, \lambda_{n+1}$ in $[0,1]$ with $\sum_{i=1}^{n+1} \lambda_i = 1$ such that
\begin{equation*}
x  = \lambda_1 y_1 + \lambda_2 y_2 + \cdots + \lambda_{n+1} y_{n+1}.
\end{equation*}

Fix $\ep > 0$ and choose $q \in \N$ sufficiently large and $r_1,r_2,\ldots,r_{n+1} \in \N$ such that  $q=\sum_{i=0}^{n+1} r_i$ and, moreover, for any $i = 1, \ldots, n+1$,
\begin{equation}
\left| \lambda_i  - \frac{r_i}{q}\right| \;\le\; \frac{\ep}{4(n+1)\beta}.
\end{equation}

It follows from  \eqref{lim:STO-ExtrR} that  there exists  $M_\ep \in \N$ such that, if $m \ge M_\ep$, then
\begin{equation}
\max_{1 \le i \le n+1}\; \max_{0 \le s \le q} \; d\left(y_i, \frac{\cR_m(Y)(\tau_{s(m+\ell)} \om)}{m} \right) \;\le\; \frac{\ep}{4}.
\label{lim:D_unif}
\end{equation}

For any $k \ge q(M_\ep + \ell)$, let $m \ge M_\ep$ be the unique integer such that $q(m+\ell) \le k < q(m+1+\ell)$ and,  for each $i=1,2,\ldots,n+1,$ and $j = 0,\ldots,r_i-1$, set $s_{ij} := r_{i-1} + j$. Let $r_0 := 0$ so that $s_{1j} = j$ for $j = 0, \ldots, r_1 - 1$. Then $d(y_i, m^{-1}\cR_m(Y,\tau_{s_{ij}(m+\ell)}\om))$ is controlled by \eqref{lim:D_unif}, which yields the existence of
$\gam_{ij} \in \cA_{0,m}(\tau_{s_{ij}(m+\ell)}\om)$
such that the end point $y_{ij} := \gamma_{ij}(m)$ satisfies
$$
\left|y_i - \frac{y_{ij}}{m}\right| \le \frac{\ep}{4}.
$$

We denote the starting point $\gam_{ij}(0)$  by $y^0_{ij}$. and next we construct an admissible path $\gamma \in \cA_{0,k}(\om)$ using the sub-paths $\gam_{ij}$, $i=1,2,\ldots,n+1$, $j = 0,1,\ldots,r_i-1$, as follows. 

Set $i = 1$. We connect the sub-paths $\gam_{1j}$, with  $j = 0,1,\ldots,r_1-1,$ to construct $\gam \in \cA_{0,r_1(m+\ell)}(\om)$. For $0\le t < m$, we set $\gamma(t) := \gamma_{10}(t)$, and for $m \le t < m+\ell$, we define $\gamma$ to be a bridge connecting $y_{10}$ to $[y_{10}] + y^0_{11}$ constructed as in Lemma \ref{lem:ODE}. Note that at $t = m + \ell$, the path is ready to be connected with $\gamma_{11}$. Hence, for $m+\ell \le t < 2m + \ell$, we define $\gamma(t) = [y_{10}] + \gamma_{11}(t - (m+\ell))$, and then for $2m+\ell \le t < 2(m+\ell)$ we build a bridge to $[y_{10}] + [y_{11}] + y^0_{12}$. It follows from the periodicity in space and the stationarity in time that $\gamma \in \cA_{0,2(m+\ell)}(\om)$. 

We repeat this procedure for a total of $r_1$ times as follows. Suppose $\gam$ is constructed on $[0,j(m+\ell)]$. Then, for $t\in [j(m+\ell),j(m+\ell)+m]$, set $\gam(t) := \sum_{l=0}^{j-1} [y_{1l}] + \gam_{1j}(t-j(m+\ell))$. Next, for $t\in [j(m+\ell)+m,(j+1)(m+\ell)]$, set $\gam$ to be the bridge that connects $\gam(j(m+\ell)+m)$ to $\sum_{l=0}^{j} [y_{1l}] + y^0_{1(j+1)}$.  In the $r_1$-th step and for $r_1(m+\ell) - \ell \le t \le r_1(m+\ell)$, $\gamma(t)$ is chosen as a bridge which connects $\gam(r_1(m+\ell)-\ell)$ and $\sum_{j=0}^{r_1-1} [y_{1j}] + y^0_{20}$. By construction, $\gamma \in \cA_{0,r_1(m+\ell)}(\om)$ and, in particular,
$
\gamma(r_1(m+\ell)) = \sum_{j=0}^{r_1-1}[y_{1j}] + y^0_{20}.
$

Now suppose that, for some $1\le i \le n$, we have constructed $\gam \in \cA_{0,S_i(m+\ell)}(\om)$ where $S_i := \sum_{p=1}^i r_p$ and, in particular, $\gam(S_i(m+\ell)) = \sum_{p=1}^i \sum_{j=0}^{r_p-1} [y_{pj}] + y^0_{(i+1)0}$. We continue the construction so that $\gam \in \cA_{0,S_{i+1}(m+\ell)}(\om)$. Since $\gam(S_i(m+\ell)) = y^0_{(i+1)0}$ modulo an element in $\Z^n$, we can connect the sub-path $\gam_{(i+1) 0}$ to $\gam$. Then, following the procedure as in the case of $i=1$, we translate and connect the sub-paths $\gam_{(i+1)j}$, $j = 0,1,\ldots, r_{i+1}-1$, to $\gam$ and obtain $\gamma \in \cA_{0,S_{i+1}(m+\ell)} (\om)$. We have, in particular,
$
\gamma(S_{i+1}(m+\ell)) = \sum_{p=1}^{i+1}\sum_{j=0}^{r_p -1} [y_{pj}] + y^0_{(i+2)0},
$
which is ready for the next step in the induction.

After $n+1$ steps, we obtain $\gamma \in \cA_{0,q(m+\ell)}$. 
We can set $y^0_{(n+2)0} := 0$ in the $(n+1)$-th step, so that $\gamma(q(m+\ell)) = \sum_{i=1}^{n+1}\sum_{j=0}^{r_i-1} [y_{ij}]$.

Finally, for $q(m+\ell) \le t < k$, we let $\gamma(t) = \gamma(q(m+\ell))$. Then, $\gamma \in \cA_{0,k}(\om)$, and 
\begin{equation*}
{x_k} := \frac{\gamma(k)}{k} = \frac{1}{k} \sum_{i=1}^{n+1} \sum_{j=0}^{r_i -1} [y_{ij}] \in \frac{\cR_k(Y)(\om)}{k}.
\end{equation*}

Let $K_\ep := \max\{q(M_\ep + \ell), [8q(\ell+1)\beta/\ep] + 1\}$ and observe that,  if $k \ge K_\ep$, then $|x_k - x| \le \ep$. Indeed, from the construction above, we have
\begin{equation*}
|x - x_k| \le \left|\sum_{i=1}^{n+1} \left(\lambda_i - \frac{r_i}{q}\right) y_i \right| + \left|\sum_{i=1}^{n+1} \frac{r_i}{q} y_i - \sum_{i=1}^{n+1}\sum_{j=0}^{r_i-1} \frac{[y_{ij}]}{k}\right|.
\end{equation*}

For the first term on the right hand side above, in view of \eqref{Dab} and $y_i \in D \subset \overline{B}_\beta(0)$, we have
\begin{equation*}
\left|\sum_{i=1}^{n+1} \left(\lambda_i - \frac{r_i}{q}\right) y_i \right|  \le \beta \sum_{i=1}^{n+1} \frac{\ep}{4(n+1)\beta} = \frac{\ep}{4}.
\end{equation*}

For the second term, we rewrite the sum as
\begin{equation*}
\sum_{i=1}^{n+1} \sum_{j=0}^{r_i-1} \left(\frac{y_i}{q} - \frac{y_{ij}}{k} + \frac{\hat{y}_{ij}}{k}\right) = \sum_{i=1}^{n+1} \sum_{j=0}^{r_i-1} \left(\frac{m}{k} \left(y_i - \frac{y_{ij}}{m}\right) + \left(\frac{1}{q} - \frac{m}{k}\right) y_i + \frac{\hat{y}_{ij}}{k}\right).
\end{equation*}
and estimate each of the three terms in the sum below.

Using that $qm < k$, we find
\begin{equation*}
\sum_{i=1}^{n+1} \sum_{j=0}^{r_i-1} \frac{m}{k} \left|y_i - \frac{y_{ij}}{m}\right| \le  \sum_{i=1}^{n+1} \sum_{j=0}^{r_i-1} \frac{m}{k} \frac{\ep}{4} = \frac{\ep}{4} \frac{qm}{k} < \frac{\ep}{4},
\end{equation*}
the fact that $y_i \in D \subset \ol{B}_\beta$ yields 
\begin{equation*}
\sum_{i=1}^{n+1} \sum_{j=1}^{r_i-1} \left|\left(\frac{1}{q} - \frac{m}{k}\right) y_i\right| \le \beta \sum_{i=1}^{n+1} \sum_{j=1}^{r_i-1} \left(\frac{1}{q} - \frac{m}{k}\right) = \beta\left(1-\frac{qm}{k}\right)
\le \frac{\beta q(\ell+1)}{k} \le \frac{\ep}{8},
\end{equation*}
and, finally and for the third term, since $\hat{y}_{ij} \in Y \subset \overline{B}_{\sqrt{n}}$ we have
\begin{equation*}
\sum_{i=1}^{n+1} \sum_{j=1}^{r_i-1} \left|\frac{\hat{y}_{ij}}{k}\right| \le  \frac{q\sqrt{n}}{k} \le \frac{\beta q(\ell+1)}{k} \le \frac{\ep}{8}.
\end{equation*}

Combining these estimates above, we establish \eqref{STO-s1}.

{\it Step 2: Convergence in Hausdorff metric.} Since we work with an fixed $\om \in \widetilde\Om$, for simplicity, we omit its dependence. 

We prove that $D$ is the long time average of $\cR_m(Y)/m$, $m \in \N$, that is
\begin{equation}\label{lim:Rm}
\lim_{m \to \infty} \frac{\cR_m(Y)}{m} = D \quad \text{in} \quad (\srC, \rho).
\end{equation}

Let $f_m : D \to \R$ be defined by
\begin{equation*}
f_m (x) := d\left(x, \frac{\cR_m(Y)}{m}\right) = \inf_{y \in \cR_m(Y)/m} d(x,y).
\end{equation*}

In view of \eqref{STO-s1}, $f_m(x)\to 0$ for each $x\in D$. Since the sequence $(f_m)_{m\in \N}$ is uniformly bounded and equicontinuous in $D$, 
it follows that actually $f_m \to 0$ uniformly in $D$, that is
\begin{equation}\label{s2-c1}
\lim_{m \to \infty}\quad \sup_{x \in D} \ \inf_{y \in \cR_m(Y)/m} d(x,y) = 0.
\end{equation}
 
Next we use \eqref{lim:coR} and the fact that $\cR_m(Y) \subset \co \cR_m(Y)$ for all $m \in \N$ to get
\begin{equation}\label{s2-c2}
0 \,\le\, \limsup_{m \to \infty} \quad \sup_{x \in \cR_m(Y)/m}\  \inf_{y \in D} d(x,y) \,\le\, \limsup_{m \to \infty} \quad \sup_{x \in \co \cR_m(Y)/m}\  \inf_{y \in D} d(x,y) \,=\, 0,
\end{equation}

Finally, \eqref{s2-c1} and \eqref{s2-c2} together confirm the claim \eqref{lim:Rm}.

{\it Step 3: From discrete to continuous time.} We prove that $D$ is the long time limit of $\cR_t(Y)/t$ for $t \in \R$. 

Recall that,  for any $A \in \srC$ and $c \in \R$,
\begin{equation}\label{cA}
\rho(cA, A) \le |c-1| \|A\|.
\end{equation}

This follows from the observation that, for any $x \in cA$, there exists $x' \in A$ such that $x = cx'$ and, since $|x-x'| = |cx'-x'| \le |c-1|\|A\|$,
\begin{equation*}
\sup_{x \in cA} \inf_{y \in A} |x-y|\le |c-1|\cdot \|A\|,
\end{equation*}
while, for any $x \in A$, there exists $x' \in cA$ such that $x' = cx$ and, since $|x-x'| = |cx-x| \le |c-1| \|A\|$,
\begin{equation*}
\sup_{x \in A} \inf_{y \in cA} |x-y| \le |c-1|\|A\|.
\end{equation*}

Also the monotonicity of $\cR_t(Y)$ with $t>0$  yields 
\begin{equation*}
\cR_{[t]}(Y) \subset \cR_t(Y) \subset \cR_{[t]+1}(Y).
\end{equation*}

It follows that
\begin{equation*}
\sup_{x \in \cR_t(Y)} \inf_{y \in tD} |x-y| \le \sup_{x \in \cR_{[t]+1}(Y)} \inf_{y \in tD} |x-y| \le \quad \rho(\cR_{[t]+1}(Y), tD),
\end{equation*}
and
\begin{equation*}
\sup_{x \in tD} \inf_{y \in \cR_t(Y)} |x-y| \le \quad \sup_{x \in tD} \inf_{y \in \cR_{[t]}(Y)} |x-y| \le \quad \rho(\cR_{[t]}(Y), tD)
\end{equation*}
and, thus,
\begin{equation}\label{s3-c1}
\rho(\cR_t(Y),tD) \le \max \left\{\rho(\cR_{[t]}(Y), tD), \rho(\cR_{[t]+1}(Y), tD)\right\}.
\end{equation}

Using the triangle inequality and \eqref{cA}, we also get
\begin{equation}\label{s3-c2}
\rho(\cR_{[t]}(Y),tD) \le \rho(\cR_{[t]}(Y),[t]D) + \rho([t]D,tD) \le \rho(\cR_{[t]}(Y),[t]D) + \|D\|,
\end{equation}
and, as above,
\begin{equation}\label{s3-c3}
\rho(\cR_{[t]+1}(Y),tD) \le \rho(\cR_{[t]+1}(Y),([t]+1)D) + \|D\|.
\end{equation}

Finally, in view of  the positive homogeneity of $\rho$ and \eqref{s3-c1}, \eqref{s3-c2}, and \eqref{s3-c3}, we have 
\begin{equation*}
\rho\left(\frac{\cR_t(Y)}{t},D\right) \le \max \left\{ \frac{[t]}{t} \rho\left(\frac{\cR_{[t]}(Y)}{[t]}, D\right), \frac{[t]+1}{t} \rho\left(\frac{\cR_{[t]+1}(Y)}{[t]+1}, D\right) \right\} + \frac{\|D\|}{t} \xrightarrow{t \to \infty} 0.
\end{equation*}
\end{proof}



In this somewhat more standard setting, $a(x,t)$ is $\Z^{n+1}$-periodic, that is, for $(x,t) \in \R^n \times \R$, and $(k,l)\in \Z^n\times \Z$,
\begin{equation*}
a(x+k,t+l)=a(x,t),
\end{equation*}
the proof simplifies considerably. Indeed this is a special case of the general setting of Section \ref{sec:prelim} and corresponds to $\Om$ having one single element. It is then easy to see that, for any $x \in \R^n$, $k \in \N$ and $t \ge 0$,  $\cR_t(x) = [x] + \cR_t(\hat{x})$ and $\cR_{t+k}(x,k) = \cR_t(x)$. 
Moreover, as in Lemma \ref{lem:STO-subadd}, the family $(\cR_m(Y))_{m \in \N}$ is almost subadditive, in the sense that, for all $m, k \in \N$ with $k < m$,
\begin{equation}
\cR_m(Y) \subset \cR_k(Y) + \cR_{m-k}(Y)+\tilde Y.
\end{equation}

It follows that $\lim_{m \to \infty} \co \cR_m(Y)/m = D$ in $(\srC, \rho)$. As before, we can prove that $D$ is also the limit of $\cR_m(Y)/m$, once we establish \eqref{STO-s1}. The proof of the latter is much simpler in the space-time periodic setting. We do not need the technical Lemma \ref{lem:STO-coR} which, in fact, is no different than Theorem \ref{thm:XR}. Indeed, for any $x \in D$ which is a convex combination of $(y_i \in \cE(D))_{i=1,\ldots,n+1}$, to construct $\gam \in \cA_{0,k}$ as in the proof of Theorem \ref{thm:STO-R} above, we only use one sub-path $\gam_i  \in \cA_{0,m}$ for each $y_i$, and then copy, translate and connect them. In the temporal random setting, the environment does not simply repeat itself and, for each $y_i$, we had to find a sequence of sub-paths $(\gam_{ij} \in \cA_{0,m}(\tau_{s_{ij}(m+\ell)}\om))_{0\le j \le r_i-1}$, where $\tau_{s_{ij}(m+\ell)}\om$ takes care of the change of the environment.

\section{The proofs of the main results} \label{sec:hom}


We prove Theorem \ref{thm:STO-R0} first and demonstrate how the large time average of the enlarged reachable set $\cR_t(Y)(\om)$ controls that of reachable set $\cR_t(x)(\om)$ from any point. Then  we apply Theorem \ref{thm:STO-R0} to prove the homogenization theory for the level-set equation \eqref{PDE1}.

\subsection*{Large time average of reachable set from a point} 

Theorem \ref{thm:STO-R0} says, essentially, that the large time average of the reachable set $\cR_t(x)(\om)$ converges in $(\srC, \rho)$ uniformly in $Y$. In view of \eqref{Rx_rand}, this convergence is in fact local uniform in  $\R^n$.

\begin{proof}[Proof of Theorem \ref{thm:STO-R0}]
Set $\widetilde\Om$ be as in Lemma \ref{lem:STO-coR} and $\ell=[\sqrt{n}/\alpha]+1$ as in Lemma \ref{lem:ODE}. For any $\om \in \widetilde\Omega$ and any $x \in Y$, in view of \eqref{Rt-base}, we have  $Y \subset \overline{B}_{\sqrt{n}}(x) \subset \cR_{\ell}(x)(\om)$. Moreover, the definition of $\cR_t$ and \eqref{Rt_rand}, for $t > \ell$, yield
\begin{equation*}
\cR_t(x)(\om) = \bigcup_{y \in \cR_{\ell}(x)(\om)} \cR_t(y,\ell)(\om) = \bigcup_{y \in \cR_{\ell}(x)(\om)} \cR_{t-\ell}(y)(\tau_\ell \om) \supset \bigcup_{y \in Y} \cR_{t-\ell}(y)(\tau_\ell \om) = \cR_{t-\ell}(Y)(\tau_\ell \om),
\end{equation*}
and, thus,
\begin{equation}\label{Rset-incl}
\cR_{t-\ell}(Y)(\tau_\ell \om) \subset \cR_t(x)(\om) \subset \cR_t(Y)(\om).
\end{equation}

Using an argument similar to the one that leads to \eqref{s3-c1}, we obtain
\begin{equation}\label{Rtx-c1}
\rho\left(\cR_t(x)(\om), tD\right) \le \max\left\{ \rho(\cR_{t-\ell}(Y)(\tau_\ell \om), tD), \rho(\cR_{t}(Y)(\om), tD) \right\}.
\end{equation}

Note that estimate above holds for all $x \in Y$, while the right hand side is independent of $x$. As a result, \eqref{Rtx-c1} can be improved to
\begin{equation}\label{Rtx-c2}
\sup_{x \in Y} \rho\left(\frac{\cR_t(x)(\om)}{t}, D\right) \le \max\left\{ \rho\left(\frac{\cR_{t-\ell}(Y)(\tau_\ell \om)}{t}, D\right), \rho\left(\frac{\cR_{t}(Y)(\om)}{t}, D\right) \right\},
\end{equation}
and to conclude we only need to control each of the two terms inside the max on the right hand side.

Theorem \ref{thm:STO-R} yields that the second term converges to zero as $t \to \infty$, while for the 
the first term, using \eqref{cA} and the positive homogeneity of $\rho$, we get 
\begin{equation}
\begin{aligned}
\rho\left(\frac{\cR_{t-\ell}(Y)(\tau_\ell \om)}{t}, D\right) &\le \rho\left(\frac{\cR_{t-\ell}(Y)(\tau_\ell \om)}{t}, \frac{t-\ell}{t}D\right) + \rho\left(\frac{t-\ell}{t}D, D\right)\\
&\le \frac{t-\ell}{t} \rho\left(\frac{\cR_{t-\ell}(Y)(\tau_\ell \om)}{t-\ell}, D\right) + \frac{\ell \|D\|}{t}.
\end{aligned}
\end{equation}

As $t \to \infty$, the last term above vanishes. 
The proof of \eqref{lim:Rt-uni} is now complete, and since \eqref{lim:R0t} is a weaker statement, it follows immediately.
\end{proof}

\subsection*{The homogenization of the level-set pde}

Recall that the  Lax-Oleinik formula for the solution of \eqref{PDE-hom} yields that
\begin{equation*}
\overline u(x,t) = \inf \{u_0(y)\,:\,\frac{x-y}{t} \in D \}=\inf\{u_0(y)\,:\,x\in y+tD\}.
\end{equation*}

\begin{proof}[Proof of Theorem \ref{thm:stoch}]
Fix $\om \in \widetilde\Omega$, $T>0$ and $R>0$.
The representation formula of $u^\ep$ gives, for any $(x,t)\in B_R \times [0,T]$,
\begin{equation}\label{uep_rep}
\begin{aligned}
u^\ep(x,t,\om)&=\inf \left\{u_0(y)\,:\,\frac{x}{\ep} \in \cR_{\frac{t}{\ep}}\left(\frac{y}{\ep}\right)(\om)\right\}\\
&=\inf \left\{u_0(y)\,:\, x \in \ep \left( \left[\frac{y}{\ep}\right]+\cR_{\frac{t}{\ep}}\left(\frac{y}{\ep}-\left[\frac{y}{\ep}\right]\right)(\om)\right)\right\}.
\end{aligned}
\end{equation}

In light of \eqref{Rt-base}, for $\ep \in (0,1)$ and $t\in [0,T]$,
\begin{equation}\label{set-control}
\begin{aligned}
D_\ep(y,t,\om) &:= \ep \left( \left[\frac{y}{\ep}\right]+\cR_{\frac{t}{\ep}}\left(\frac{y}{\ep}-\left[\frac{y}{\ep}\right]\right)(\om)\right) = y + \ep \left[ \cR_{\frac{t}{\ep}} \left(\frac{y}{\ep}-\left[\frac{y}{\ep}\right]\right)(\om) -  \left(\frac{y}{\ep}-\left[\frac{y}{\ep}\right]\right) \right]\\
& \subset y+\ep \ol{B}_{\frac{t\beta}{\ep}} = y + \ol{B}_{t\beta} \subset y + \ol{B}_{T\beta},
\end{aligned}
\end{equation}
a fact  yielding  that $x\in D_\ep(y,t,\om)$ only if $|y| \leq T\beta+R$.

Next let  $K_\ep(x,t,\om) := \{y :\, x \in D_\ep(y,t,\om) \}$ and $K(x,t) := \{y \,:\, x \in y + tD\}=x-tD$. We show that
\begin{equation}\label{lim:K}
\lim_{\ep \to 0}\, \sup_{(x,t) \in B_R\times [0,T]} \, \rho \left(K_\ep(x,t,\om), K(x,t) \right) = 0.
\end{equation}

Once this limit is established, it follows that
\begin{align*}
u^\ep(x,t,\om)
=&\inf \left\{u_0(y)\,:\,   x \in D_\ep(y,t,\om)\right\}=
\inf \left\{u_0(y)\,:\, y\in \ol{B}_{T\beta+R} \quad \text{and} \quad y \in K_\ep(x,t,\om)\right\}\\
\xrightarrow{\ep \to 0}\ &\inf \left\{u_0(y)\,:\,  y\in \ol{B}_{T\beta+R} \quad \text{and} \quad y \in K(x,t)\right\}
=\inf \left\{u_0(y)\,:\,  x \in y+tD\right\}=\overline u(x,t),
\end{align*}
and, moreover, the convergence is uniform on $B_R \times [0,T]$. 

It remains to prove \eqref{lim:K}. Fix  $\delta > 0$ and consider first the case $0
\le t \le \delta$. In view of \eqref{set-control}, for all $x \in B_R$,
\begin{equation*}
\{x\} \subset K_\ep(x,t,\om) \subset \{y \,:\, x \in y+ \ol{B}_{t\beta}\} = x - \ol{B}_{t\beta}, 
\end{equation*}
and it follows that
\begin{equation*}
\sup_{y_1 \in K_\ep(x,t,\om)} \, \inf_{y_2 \in K(x,t)} |y_1-y_2| \le \rho\left( x - \ol{B}_{t\beta}, x - tD\right) = t\rho(\ol{B}_{\beta}, D) \le \delta \rho(\ol{B}_{\beta},D),
\end{equation*}
and
\begin{equation*}
\sup_{y_1 \in K(x,t)} \, \inf_{y_2 \in K_\ep(x,t,\om)} |y_1-y_2| \le \rho\left(x-tD, \{x\} \right) = t\|D\| \le \delta \|D\|.
\end{equation*}

Since the two estimates above are uniform in $x$ and $t$,  for $C := \max(\rho(\ol{B}_{\beta},D), \|D\|)$, we get
\begin{equation}\label{K-c1}
\sup_{(x,t) \in B_R \times [0,\delta]} \rho(K_\ep(x,t,\om), K(x,t)) \le C \delta.
\end{equation}

Next we consider the case $t > \delta$. Taking $\ep$ small so that $t/\ep > \ell$, we claim that, for all $x \in B_R$,
\begin{equation}\label{K-c2}
x - \ep \cR_{\frac{t}{\ep} - \ell}(Y)(\tau_\ell \om) \subset K_\ep(x,t,\om) \subset x - \ep \cR_{\frac{t}{\ep}}(Y)(\om) + \ep Y.
\end{equation}

The first inclusion follows from the observation that, if $x \in y + \ep \cR_{\frac{t}{\ep} - \ell}(Y)(\tau_\ell \om)$, then
\begin{equation*}
 x \in y + \ep \left[\cR_{\frac{t}{\ep}} \left(\frac{y}{\ep}-\left[\frac{y}{\ep}\right]\right)(\om) -  \left(\frac{y}{\ep}-\left[\frac{y}{\ep}\right]\right) \right] = D_\ep(y,t,\om),
\end{equation*}
which is a consequence of the first inclusion in \eqref{Rset-incl}.

The second inclusion in \eqref{K-c2} is due to the fact that, if $x \in D_\ep(y,t,\om)$, then
\begin{equation*}
x \in y + \ep \cR_{\frac{t}{\ep}}(Y)(\om) + \ep \tilde{Y},
\end{equation*}
which follows from the second inclusion in \eqref{Rset-incl}. Thus,
\begin{equation*}
\sup_{y_1 \in K_\ep(x,t,\om)} \, \inf_{y_2 \in K(x,t)} |y_1-y_2| \le \rho\left( x - \ep\cR_{\frac{t}{\ep}}(Y)(\om) + \ep Y, x - tD\right) = t\rho\left(\frac{\cR_{t/\ep}(Y)(\om)}{t/\ep} + \frac{\ep \tilde{Y}}{t}, D\right),
\end{equation*}
and
\begin{equation*}
\sup_{y_1 \in K(x,t)} \, \inf_{y_2 \in K_\ep(x,t,\om)} |y_1-y_2| \le \rho\left(x-tD, x - \ep\cR_{\frac{t}{\ep}-\ell}(Y)(\tau_\ell \om) \right) = t \rho \left(D, \frac{\cR_{t/\ep - \ell}(Y)(\tau_\ell \om)}{t/\ep}\right).
\end{equation*}

Again the uniformity  in $x$ and $t$ of the last two estimates are uniform in $x$ and $t$ implies 
\begin{equation}\label{K-c3}
\begin{aligned}
\sup_{(x,t) \in B_R \times (\delta,T]} \rho(K_\ep(x,t,\om), K(x,t)) \le &\ T\max\left\{ \rho\left(\frac{\cR_{t/\ep}(Y)(\om)}{t/\ep}, D\right), \rho\left(\frac{\cR_{t/\ep - \ell}(Y)(\tau_\ell \om)}{t/\ep}, D\right) \right\}\\
 & + \ep\|\tilde{Y}\|.
\end{aligned}
\end{equation}

We first let $\ep \to 0$ and apply Theorem \ref{thm:STO-R} in \eqref{K-c3} and then we let $\delta \to 0$ in \eqref{K-c1}. Combining these two estimates, we establish \eqref{lim:K} and complete the proof of the theorem.
\end{proof}

\section{Further applications}
\label{sec:applications}

In this section, we present two more  examples of Hamilton-Jacobi equations with Hamiltonians of linear growth. The first concerns front propagation in an environment that is subjected to a background drift, and the second involves a non-coercive Hamiltonian.


\subsection*{Front propagation with ambient drift.}

We study the behavior, as $\ep \to 0$, of the solution $u^\ep = u^\ep(x,t,\om)$ to
\begin{equation}\label{PDE-drift}
\begin{cases}
u^\ep_t+a\left(\frac{x}{\ep},\frac{t}{\ep}, \om\right)|Du^\ep| + b\left(\frac{x}{\ep},\frac{t}{\ep},\om\right) \cdot Du^\ep =0 \quad &\text{in} \ \R^n \times (0,\infty),\\
u^\ep =u_0 \quad &\text{on} \ \R^n \times \{0\},
\end{cases}
\end{equation}
which models front propagation with velocity $V = a\left(\frac{x}{\ep},\frac{t}{\ep},\om\right) \nu + b\left(\frac{x}{\ep},\frac{t}{\ep},\om\right)$, where $\nu$  is the normal vector to the front.

In addition to {\upshape(A)}, we assume that the random process $b(\cdot,\cdot,\om) \in C^{0,1}(\R^{n+1},\R^n)$ satisfies
\begin{enumerate}
 \item[(B1)]  $b=b(x,t,\omega)$ is $\Z^n$-periodic in $x$ and stationary in $t$ with respect to $(\tau_k)_{k\in\Z}$, that is, for every $(x,l)\in \R^n \times \Z^n$, $(t,k)\in \R \times \Z$, and $\omega\in\Omega$, 
\[
b(x+l,t,\tau_k\omega) = b(x,t+k,\omega),
\]
 
\item[(B2)] there exist $\eta>0$ such that for all $(x,t)\in \R^{n+1}$ and $\omega \in \Omega$,
\begin{equation}\label{B2}
\alpha - |b(x,t,\om)| \ge \eta.
\end{equation}
\end{enumerate}

As before, we group the above assumptions as
\begin{enumerate}
\item[(B)] $b = b(x,t,\om)$ satisfies {\upshape(B1)} and {\upshape(B2)}.
\end{enumerate}

In view of (B2), the Hamiltonian $H(x,t,p,\om) = a(x,t,\om)|p| + b(x,t,\om)\cdot p$ is coercive and 
the representation formula for the solution $u^\ep$ is 
\begin{equation}
u^\ep(x,t,\om) = \inf \left\{ u_0(y) \; : \; \frac{x}{\ep} \in \cR_{\frac{t}{\ep}} \left(\frac{y}{\ep}\right)(\om) \right\},
\end{equation}
for the reachable set $\cR_t(x)(\om)$ with admissible paths associated to the control system
\begin{equation*}
\begin{cases}
\gam'(r) = f(\gam(r),r,\xi(r)) := b(\gam(r),r) + a(\gam(r),r) \xi(r) \quad &\text{for} \ r\in (0,t),\\
|\xi(r)| \leq 1  \quad &\text{a.e.}\   r \in (0,t),
\end{cases}
\end{equation*}

The set of admissible paths in the time interval $[s,t]$ is
\begin{equation}\label{def:A-b-Om}
\cA_{s,t}(\omega) := \{\gam:[s,t] \to \R^n \,:\, |\gam'(r)-b(\gam(r),r,\om)| \leq a(\gam(r),r,\omega) \quad \text{for a.e.} \ r\in [s,t]\}.
\end{equation}

Following the arguments developed earlier, we can show that the large time average $\cR_t(x)(\om)/t$, for almost all $\om \in \Om$, converges in $(\srC,\rho)$, to some compact and convex subset $D$ of $\R^n$. Indeed, thanks to (B2), $\ol{B}_{t\eta}(x) \subset \cR_t(x)(\om) \subset \ol{B}_{t(\beta+\alpha-\eta)}$ (this is the analogue of \eqref{Rt-base}) and, furthermore, any two points $y_1, y_2 \in Y$ can be connected by an admissible path within time $\ell' := [\sqrt{n}/\eta]  + 1$ (this is an analogue of Lemma \ref{lem:ODE}). On the other hand, due to the periodicity in space and stationary ergodicity in time of the control system, the translation rules  and the (modified) subadditivity of $\cR_t(x)(\om)$, that is Lemma \ref{lem:R_rand} and Lemma \ref{lem:STO-subadd}, still hold. Hence, we can carry out the whole program of the analysis in this paper aand obtain the following homogenization result.  Its  proof is exactly the same as those of Theorem \ref{thm:STO-R0} and Theorem \ref{thm:stoch} and hence is omitted.

\begin{thm}\label{thm:STO-drift}
Assume {\upshape(A)} and {\upshape(B)}. There exists a compact and convex  $D\subset \R^n$ and an event $\widetilde \Omega \in \cF$ of full probability such that, for each $\om \in \widetilde\Omega$:
\begin{enumerate}
\item[\upshape (i)] 
The large time average of reachable set converges locally uniformly to $D$, that is, as $t\to \infty$, 
\begin{equation}\label{lim:Rt-drift-uni}
\lim_{t\to \infty} \sup_{x\in Y} \rho\left(\frac{\cR_t(x)(\om)}{t},D\right)=0.
\end{equation}
\item[\upshape (ii)] As $\ep \to 0$,  the solution $u^\ep=u^\ep(\cdot,\cdot,\omega)$ of \eqref{PDE-drift} converges locally uniformly in $\R^n \times [0,\infty)$ to  $\overline u$  the solution to  \eqref{PDE-hom} with $\ol{H}$ defined as in \eqref{Hbar-def}.
\end{enumerate}
\end{thm}

\subsection*{A non-coercive Hamilton-Jacobi equation.}

We are interested in the limit, as $\ep \to 0$, of the solution $v^\ep = v^\ep(x',t,\om)$ of 
\begin{equation}\label{APP}
\begin{cases}
v^\ep_t+a\left(\frac{x'}{\ep}\right)|D_xv^\ep|+v^\ep_{x_{n+1}}=0 \quad &\text{in} \ \R^{n+1} \times (0,\infty)\\
v^\ep =v_0 \quad &\text{on} \ \R^{n+1} \times \{0\}.
\end{cases}
\end{equation}
Here $x'=(x,x_{n+1})\in \R^n \times \R$ and $a\in C^{0,1}(\R^{n+1})$ satisfies {\upshape(A)}, with $x_{n+1}$ playing the role of $t$.
We also write $p',q', y'\in \R^{n+1}$ as $p'=(p,p_{n+1}), q'=(q,q_{n+1}), y'=(y,y_{n+1})\in \R^n \times \R$. 

For $(y',p')\in \R^{n+1}\times \R^{n+1}$, the Hamiltonian is
\[
H'(y',p')=a(y')|p|+p_{n+1},
\]
and,  for $(q',y')\in \R^{n+1} \times \R^{n+1}$, the  Lagrangian  is 
\[
L'(y',q')=\begin{cases}
0, \qquad &\text{for}\ q_{n+1}=1, \ \text{and}\ |q| \leq a(y'),\\
+\infty, \qquad &\text{otherwise.}
\end{cases}
\]

The representation formula for the solution $v^\ep$ is 
\begin{equation*}
\begin{aligned}
v^\ep(x',t,\om) &= \inf_{y \in \R^n} \Big\{ v_0(y,x_{n+1}-t) \;:\; \text{there exists} \ \gam:\left[0,\ep^{-1}t\right] \to \R^n \ \text{such that } \gam(0)=\frac{y}{\ep},\\
&\qquad \qquad \qquad \ \gam\left(\frac{t}{\ep}\right)=\frac{x}{\ep}, \text{ and } |\dot{\gam}(s)| \leq a\left(\gam(s),\frac{x_{n+1}-t}{\ep}+s\right) \text{ for a.e } s \in \left(0,\ep^{-1}t\right) \Big\}.
\end{aligned}
\end{equation*}

Recalling the definition of the reachable set $\cR_t(x,s)(\om)$ in Section \ref{sec:prelim}, with the associated control system \eqref{control_syst}, we rewrite the formula as
\begin{equation*}
v^\ep(x',t,\om) = \inf_{y \in \R^n} \left\{ v_0(y,x_{n+1}-t) \;:\; \frac{x}{\ep} \in \cR_{\frac{t}{\ep} + \left(\frac{x_{n+1}-t}{\ep}\right)} \left(\frac{y}{\ep}, \frac{x_{n+1}-t}{\ep}\right)(\om) \right\}.
\end{equation*}

We have the following homogenization result for $v^\ep$.

\begin{thm}\label{thm:homAPP}
Assume {\upshape(A)}. Let $\widetilde\Omega$ be as defined in Theorem \ref{thm:STO-R0}. Then, for each $\om \in \widetilde\Om$, the  solution $v^\ep$ of \eqref{APP} converges locally uniformly in $\R^{n+1} \times (0,\infty)$ to the solution $\overline v$ of 
\begin{equation}\label{APP-hom}
\begin{cases}
\overline v_t+\ol{H}(D_x \overline v)+\overline v_{x_{n+1}}=0 \quad &\text{in} \ \R^{n+1} \times (0,\infty)\\
\overline v =v_0 \quad &\text{on} \ \R^{n+1} \times \{0\},
\end{cases}
\end{equation}
where $\ol{H}$ is defined in \eqref{Hbar-def}.
\end{thm}

We only give a sketch of proof. Since the complete argument follows exactly the proof of Theorem \ref{thm:stoch}, here we only give a sketch of proof.  Comparing the representation formula of $v^\ep$ with that of $u^\ep$ in \eqref{uep_rep}, we observe that the only complication in the former is the presence of the initial time $\frac{x_{n+1}-t}{\ep}$ in the reachable set. The difficulty caused by this can be overcome, since we know precisely how the reachable sets change with respect to integral translations in time, and we can control the difference between non-integral translations and their nearest integral ones. It follows, from the proof of Theorem \ref{thm:stoch}, that, locally uniformly for $(x',t) \in \R^{n+1}\times (0,\infty)$,
\begin{equation*}
\rho\left( \left\{y \in \R^n \;:\; x \in \ep \cR_{\frac{t}{\ep} + \left(\frac{x_{n+1}-t}{\ep}\right)} \left(\frac{y}{\ep}, \frac{x_{n+1}-t}{\ep}\right)(\om)\right\}, x - tD \right) \longrightarrow 0.
\end{equation*}
This shows essentially that $v^\ep$ converges,  locally uniformly,  to $\overline v(x',t) = \inf_{x - tD} v_0(y,x_{n+1}-t)$, the solution of \eqref{APP-hom}.


\bibliographystyle{acm} 
\bibliography{jst}

\end{document}